\documentclass[12pt]{amsart}

\usepackage{amssymb}

\newcommand{\ol}{\overline}

\newcommand{\al}{\alpha}
\newcommand{\bt}{\beta}

\newcommand{\dt}{\delta}

\newcommand{\gm}{\gamma}
\newcommand{\vf}{\varphi}

\newcommand{\Om}{\Omega}

\newcommand{\Dt}{\Delta}

\newcommand{\Tht}{\Theta}

\newcommand{\const}{\mathop{\rm Const}}
\newcommand{\Id}{\mathop{\rm Id}}

\newcommand{\tm}{\times}

\newcommand{\sbs}{\subset}

\newcommand{\wdt}{\widetilde}

\newcommand{\iy}{\infty}

\newcommand{\bR}{\mathbb{R}}

\newcommand{\bZ}{\mathbb{Z}}
\newcommand{\bT}{\mathbb{T}}
\newcommand{\bC}{\mathbb{C}}
\newcommand{\bQ}{\mathbb{Q}}
\newcommand{\bD}{\mathbb{D}}

\newcommand{\calM}{\mathcal{M}}
\newcommand{\calU}{\mathcal{U}}
\newcommand{\calP}{\mathcal{P}}

\newcommand{\mbP}{\mathbf{P}}

\newcommand{\mba}{\mathbf{a}}
\newcommand{\mbU}{\mathbf{U}}
\newcommand{\mbo}{\mathbf{0}}
\newcommand{\mbl}{\mathbf{1}}

\newtheorem{theorem}{\bf  Theorem}
\newtheorem{lemma}{\bf  Lemma}

\newtheorem{corollary}{\bf \sc Corollary}

\begin{document}


\begin{center}
An approximation of Daubechies wavelet matrices by perfect
reconstruction filter banks with rational coefficients
\\[5mm]

 L.~Ephremidze, A.~Gamkrelidze
  and  E.~Lagvilava
        \end{center}

\vskip+0.5cm

 \noindent {\small {\bf Abstract.} It is described how the coefficients of Daubechies
 wavelet matrices can be approximated by rational numbers in such a way that the
 perfect reconstruction property of the filter bank be preserved exactly.}

 \vskip+0.2cm\noindent  {\small {\bf Keywords:} Daubechies wavelets,
paraunitary matrix polynomials}

\vskip+0.2cm \noindent  {\small {\bf  AMS subject classification
(2010):} 42C40}

\vskip+0.5cm

\section{Introduction} Daubechies wavelet matrices $D_N$ are
perfect reconstruction orthogonal filter banks to which there
correspond the orthonormal bases of compactly supported wavelet
functions $\{2^{-\frac i2}\psi(2^ix-j)\}_{i,j\in\mathbb{Z}}$.
However, in most of practical applications of these wavelets, what
matters is the coefficients of $D_N$ and not the form of the
corresponding function $\psi$. In mere approximation of the
irrational coefficients of $D_N$ by rational numbers, which is
desirable in  order to simplify the related calculations on a
digital computer, the perfect reconstruction property of the
filter bank $D_N$ (which is its most important property) is not
preserved in general. Obtained in this way $\hat{D}_N\approx D_N$
is a perfect reconstruction orthogonal filter bank only
approximately.  In the present paper, we describe a procedure of
approximation of Daubechies wavelet matrices $D_N$ by filter banks
$\hat{D}_N$ with rational coefficients which have the perfect
reconstruction property exactly. This approach depends on a recent
parametrization of compact wavelet matrices \cite{Lag1} which was
developed in parallel with a new matrix spectral factorization
method \cite{JanLagSt}, \cite{IEEE}.\footnote{This method is
currently patent pending.}

The paper is organized as follows. The necessary notation and
definitions are introduced in the next section. In Section III, an
exact formulation  of the problem solved is given. In Section IV,
the mathematical background of the proposed method is provided and
the method itself is described in Section V. Some results of
numerical simulations are presented in Section VI.

\section{Notation and Basic Definitions}

The sets of integer, rational, real and complex numbers are
denoted by $\bZ$, $\bQ$, $\bR$, and $\bC$, respectively.
$\bD:=\{z\in\bC:|z|<1\}$ and $\bT:=\partial\bD$. $\delta_{ij}$
stands for the  Kronecker delta, $\dt_{ij}=1$ if $i=j$ and $0$
otherwise, and $\Id_X$ is the identity map on a set $X$.

Let $\calM_F(m\tm N)$ be the set of $m\tm N$ matrices with entries
from a field $F$ (if $F$ is omitted it is always assumed that
$F=\bC$). A row of coefficients $(c_0,c_1,\ldots,c_{N-1})$ will be
sometimes called a {\em filter} and an $m\tm N$ matrix will be
called an $m${\em-channel filter bank} (with $N$ {\em taps}).

For $M\in \calM(m\tm N)$, let $\ol{M}$ be the matrix with
conjugate entries and $M^*=\ol{M}^T$.

$U\in \calM(m\tm m)$ is called unitary if $UU^*=U^*U=I_m$ where
$I_m$ stands for the $m\tm m$ identity matrix, and the set of
unitary matrices is denoted by $\calU(m)$.

$\calP[F]$ denotes the set of Laurent polynomials with
coefficients from a field $F$, and $\calP_N[F]:=\{\sum_{k=-N}^N
c_kz^k:c_k\in F, k=-N,\ldots,N\}$. If we write just $\calP$, the
field of coefficients will be clear from the context, in most
cases $\calP=\calP[\bC]$. \; $\calP^+\sbs\calP$ is the set of
polynomials (with non-negative powers of $z$, $\sum_{k=0}^N
c_kz^k\in \calP^+$) and $\calP^-\sbs\calP$ is the set of Laurent
polynomials with negative powers of $z$, $\sum_{k=1}^N
c_kz^{-k}\in \calP^-$. We emphasize that constant functions belong
only to $\calP^+$ so that $\calP^+\cap \calP^-=\emptyset$. Let
also $\calP^\pm_N=\calP^\pm\cap\calP_N$.

For power series $f(z)=\sum_{k=-\iy}^\iy c_k z^k$ and $N\geq 1$,
let $[f(z)]^-$, $[f(z)]^+$, $[f(z)]^-_N$, and $[f(z)]^+_N$,
denote, respectively, $\sum_{k=-\iy}^{-1} c_k z^k$,
$\sum_{k=0}^\iy c_k z^k$, $\sum_{k=-N}^{-1} c_k z^k$, and
$\sum_{k=0}^N c_k z^k$ and we assume the corresponding functions
under these expressions if the convergence domains of these power
series are known.

$\calP(m\tm n)$ denotes the set of $m\tm n$ (polynomial) matrices
with entries from $\calP$, and  the sets $\calP^+(m\tm n)$,
$\calP^-_N[F](m\tm n)$, etc. are defined similarly. The elements
of these sets $\mbP(z)=[p_{ij}(z)]$ are called polynomial matrix
functions. When we speak about continuous maps between these sets,
we mean that they are equipped with a usual topology.

For $p(z)=\sum_{k=-N}^N c_kz^k \in\calP$, let
$\wdt{p}(z)=\sum_{k=-N}^N \ol{c}_kz^{-k}$ and for
$P(z)=[p_{ij}(z)]\in \calP(m\tm n)$ let
$\wdt{P}(z)=[\wdt{p}_{ij}(z)]^T\in\calP(n\tm m)$.  Note that
$\wdt{P}(z)=(P(z))^*$ when $z\in\bT$. Thus usual relations for
adjoint matrices like
$\wdt{P_1+P_2}(z)=\wdt{P}_1(z)+\wdt{P}_2(z)$,
$\wdt{P_1P_2}(z)=\wdt{P}_2(z)\wdt{P}_1(z)$, etc. hold.

A polynomial matrix function $\mbU(z)\in \calP(m\tm m)$ is called
{\em paraunitary} if $\mbU(z)\wdt{\mbU}(z)=I_m$ for each
$z\in\bC\backslash\{0\}$, and the set of all paraunitary
polynomial matrices is denoted by $\calP\calU(m)$. Note that if
$\mbU\in \calP\calU(m)$, then $\mbU(z)\in \calU(m)$ for each
$z\in\bT$.

An $S\in\calP_N(m\tm m)$ is called positive definite if $S(z)$ is
such ($XS(z)X^*>0$ for each $0\not=X\in\calM(1\tm m)$) for almost
every $z\in\bT$. The polynomial matrix spectral factorization
theorem (see e.g \cite{Uniq}, \cite{poly}) asserts that every
positive definite $S\in\calP_N(m\tm m)$ can be factorized as
\begin{equation}
 S(z)=S_+(z)\wdt{S_+}(z),\;\;\;z\in\bC\backslash \{0\},
 \label{PSF}
 \end{equation}
where $S_+\in\calP^+_N(m\tm m)$ and $\det S_+(z)\not=0$ for each
$z\in \bD$. The representation (\ref{PSF}) is unique in a sense
that if $ S(z)=S_+'(z)\wdt{S_+'}(z)$, then there exists
$U\in\calU(m)$ such that $S_+(z)=S_+'(z)U$.

 A matrix $A\in \calM_F\big(m\tm mN\big)$,
\begin{equation}
 \label{wm}
A=\left(\begin{matrix}\mba^0\\ \mba^1\\\vdots\\ \mba^{m-1}
\end{matrix}\right)=
\left(\begin{matrix} a^0_{0}&a^0_{1}&\cdots&a^0_{mN-1}\\[1mm]
a^1_{0}&a^1_{1}&\cdots&a^1_{mN-1}\\
\vdots&\vdots&\vdots&\vdots\\
a^{m-1}_{0}&a^{m-1}_{1}&\cdots&a^{m-1}_{mN-1}\\
\end{matrix}\right),
 \end{equation}
is said to be a {\em wavelet matrix} of {\em rank} $m$ and {\em
genus} $N$, $A\in WM(m,N,F)$ (see \cite[p. 41]{RW}) if the shifted
versions of the rows of $A$ by arbitrary multiples of $m$ form an
orthogonal set, that is
\begin{equation}
 \label{qc}
\sum_{k=0}^{mN-1}a^i_{k+mr}\ol{a^j_{k+ms}}=m\dt_{ij}\dt_{rs},\;\;\;
r,s\in\bZ,\;\;i,j\in\{0,1,\ldots,m-1\}
 \end{equation}
(it is assumed that $a^i_k=0$ whenever $k<0$ or $k\geq mN$) and
\begin{equation}
 \label{lc}
\sum_{k=0}^{mN-1}a^i_{k}=m\dt_{i0}, \;\;\;0\leq i<m.
 \end{equation}
The conditions (\ref{qc}) and (\ref{lc}) are referred to as the
{\em quadratic} and {\em linear} conditions, respectively,
defining a wavelet matrix.

Wavelet matrices with genus 1 (and rank $m$) are called {\em Haar
wavelet matrices}, $H(m,F):=WM(m,1,F)$. It can be shown that the
first row of any $H\in H(m,F)$ consists of just 1s (see
\cite[Lemma 4.4.2]{RW}) and if $H_1,H_2\in H(m,\bC)$, then
$H_1=\left(\begin{matrix}1&0\\0&U\end{matrix}\right)H_2$ where
$U\in \calU(m-1)$ (see \cite[Corollary 4.4.3]{RW}). Consequently,
the only real Haar wavelet matrix of rank 2 with determinant 1 is
\begin{equation}
 \label{H2}
H_2=\begin{pmatrix} 1&1\\
-1&1\end{pmatrix}.
 \end{equation}

For the matrix (\ref{wm}), we consider the Fourier series (the
$z$-transform) of its rows
\begin{equation}
 \label{sf}
h_s(z)=\sum_{k=0}^{mN-1}a^s_k z^k,\;\;\;\;\;s=0,1,\ldots,m-1,
 \end{equation}
 and the corresponding {\em polyphase matrix polynomial} $A(z)\in\calP^+_F(m\tm m)$ defined by
 \begin{equation}
 \label{LM}
 A(z)=\sum_{k=0}^{N-1}A_k z^k
 \end{equation}
where $A_k=[a^i_{km+j}]\in \calM_F(m\tm m)$, $k=0,1,\ldots,N-1$,
e.g. $
A_0=\left(\begin{matrix} a^0_{0}&\cdots&a^0_{m-1}\\
\vdots&\vdots&\vdots\\
a^{m-1}_{0}&\cdots&a^{m-1}_{m-1}\\
\end{matrix}\right)
$.  We will heavily use the fact that the quadratic condition
(\ref{qc}) is equivalent to the condition on (\ref{LM}),
\begin{equation}
 \label{Pc}
 A(z)\wdt{A}(z)=mI_m,
 \end{equation}
i.e. $A(z)$ is a constant multiplier of paraunitary matrix
function. This equivalence can be checked by direct computations
(see \cite[p. 43]{RW}). Consequently, $A(1)$ is always a Haar
wavelet matrix (see \cite[p. 49]{RW}).

It is  well known as well that the quadratic condition (\ref{qc})
is also equivalent to the following condition on (\ref{sf}) (see
\cite[p. 96]{RW})
\begin{equation}
 \label{sf1}
 \sum_{k=0}^{m-1}h_r(zz_0^k)\wdt{h_{s}}(zz_0^k)=m^2\dt_{rs}.
 \end{equation}

In signal processing applications, if we split a function (signal)
$f:\bZ\to\bC$ into $m$ parts
\begin{equation}
 \label{prp}
f_r=\sum_{s=-\iy}^\iy \frac1m \langle f,\mba^r_{sm}\rangle
\mba^r_{sm},\;\;\;r=0,1,\ldots,m-1,
 \end{equation}
where $\mba^r_{s}(\cdot)=\mba^r(\cdot-s)$, $s\in\bZ$, (see
(\ref{wm})) and $\langle f,\mba^r_{s}\rangle=\sum_{k=-\iy}^\iy
f(k)\ol{\mba}^r_{s}(k)$, (it is assumed that $\mba^r(k)=0$
whenever $k$ is outside the range $\{0,1,\ldots,mN-1\}$, so that
only finitely many products in the above sums differ from $0$)
which corresponds to the filtering by each  of the rows $\mba^r$,
$r=0,1,\ldots,m-1$, followed by downsampling with rate $m$, then
 each of the equivalent conditions (\ref{qc}), (\ref{Pc}), and (\ref{sf1})
 guarantees that $f$ can be reconstructed exactly as follows
 (see \cite[Theorem 4.4.23]{RW})
\begin{equation}
 \label{prp1}
f_r=\sum_{r=0}^{m-1} f_r.
 \end{equation}
For this reason, a wavelet matrix is a {\em perfect
reconstruction} filter bank.

The linear condition (\ref{lc}) implies that a constant signal
$f:\bZ\to\mathbb{C}$ emerges from the first filter in the
representation (\ref{prp1}).

It is said that a wavelet matrix (\ref{wm}) has a {\em
polynomial-regularity degree} $d$ if
\begin{equation}
 \label{polrp}
\sum_{k=0}^{mN-1}k^p a^i_k=0,
\;\;p=0,1,\ldots,d,\;\;i=1,2,\ldots,m-1.
 \end{equation}
 The higher this degree, the more zero coefficients appear in the
 representation (\ref{prp}) of a smooth signal $f$. Note that every wavelet
 matrix has a polynomial regularity degree equal at least to $0$.

\section{Formulation of the problem}

The Daubechies wavelet matrix $D_N$ (with $2N$ taps) is the
two-channel filter bank with real coefficients
\begin{equation}
 \label{Dwm}
D_N=\begin{pmatrix} h_0\\h_1\end{pmatrix}=
\begin{pmatrix} a_0&b_0&a_1&b_1&\ldots&a_{N-1}&b_{N-1}\\
-b_{N-1}&a_{N-1}&-b_{N-2}&a_{N-2}&\ldots&-b_0&a_0
\end{pmatrix}
\end{equation}
which together with quadratic and linear  conditions (cf.
(\ref{sf1}) and (\ref{lc}))
\begin{equation}
 \label{Dqc}
|h_0(z)|^2+|h_0(-z)|^2=4\;\;\text{ when } |z|=1
 \end{equation}
and
\begin{equation}
 \label{Dlc}
h_0(1)=2;\;\;\;h_1(1)=0
 \end{equation}
where $h_j(z)=\sum_{k=0}^{2N-1}h_j[k]z^k$, $j=0,1$, has the
polynomial-regularity degree $N-1$ (\ref{polrp})

\begin{equation}
 \label{Dpolrp}
\sum_{k=0}^{2N-1}h_1[k]\cdot k^p=0 \text{ for }p=0,1,\ldots,N-1.
 \end{equation}
The way of construction of such matrices $D_N$ (computation of
coefficients in (\ref{Dwm})) was first established by Daubechies
\cite{Dau} and  is described in most books on wavelets
\cite{DauB}, \cite{Mallat}, \cite{RW}. To each matrix $D_N$ there
corresponds the Daubechies wavelet $\psi=\psi_N$ which is a
supported in $[-N+1,N]$ continuous function of certain smoothness
(depending on $N$) such that the system $\{2^{-\frac
i2}\psi(2^ix-j)\}$, ${i,j\in\mathbb{Z}}$, forms an orthonormal
basis of $L_2(\bR)$. However the forms of wavelet functions
$\psi_N$ are mostly of theoretical interest, while the numerical
values of the coefficients of $D_N$ are very important for
applications. Since they are irrational numbers in general, during
the actual calculations on digital computers, these coefficients
are quantized and thus $D_N$ is approximated by $\hat{D}_N$. It
may then happen that $\hat{D}_N$ satisfies the quadratic condition
(\ref{Dqc}) only approximately. As it has been explained in the
preceding section, the quadratic condition on a wavelet matrix
determines the perfect reconstruction property of a filter bank.

In the present paper, we propose a method of approximation of
$D_N$ by $\hat{D}_N$ which has rational coefficients and satisfies
the quadratic and linear conditions (\ref{Dqc}) and (\ref{Dlc})
exactly. It is obvious that $\hat{D}_N$ will have the maximal
polynomial regularity property (\ref{Dpolrp}) only approximately.

 \section{Mathematical Background of the Method}

 The following theorem, which plays a crucial role in the
 established method, was actually proved in \cite{JanLagSt}. We
 present here the simplified proof of this theorem.

 \begin{theorem}
Let $N\geq 1$. For any $0\not\equiv\vf\in\calP^-_N$, there exists
a unique pair of functions $\al,\bt\in \calP^+_N$ such that
\begin{gather}
 \label{Zt}
 \al(z)\wdt{\al}(z)+\bt(z)\wdt{\bt}(z)=1,\;\;\;z\in\bC\backslash \{0\}\\
 \label{Zt1} \al(1)=1;\;\;\;\bt(1)=0,
 \end{gather}
and
\begin{equation}
 \label{Zmt}
\begin{pmatrix} 1&0\\ \vf&1\end{pmatrix}
\begin{pmatrix} \al&\bt\\ -\wdt{\bt}&\wdt{\al}\end{pmatrix}\in\calP^+(2\tm
2).
 \end{equation}
Moreover, $\al$ and $\bt$ satisfy the condition
 \begin{equation}
 \label{0shi}
|\al(0)|+|\bt(0)|>0.
 \end{equation}
 \end{theorem}

 \begin{lemma}
Let $(\ref{Zmt})$ be satisfied for $\vf\in\calP^-_N$ and
$\al,\bt\in \calP^+_N$.  Then
\begin{equation}
 \label{Const}
\al(z)\wdt{\al}(z)+\bt(z)\wdt{\bt}(z)=\const,\;\;\;z\in\bC\backslash
\{0\}
 \end{equation}
 \end{lemma}
 Note that this constant should be positive (for
 $\al,\bt\not\equiv0$) since
 $\al(z)\wdt{\al}(z)+\bt(z)\wdt{\bt}(z)=|\al(z)|^2+|\bt(z)|^2$ for each
 $z\in\bT$.
\begin{proof}
It follows from (\ref{Zmt}) that
\begin{equation}
 \label{syst}
\begin{cases}
\vf\al-\wdt{\bt}=:\Phi_1\in\calP^+\\
\vf\bt+\wdt{\al}=:\Phi_2\in\calP^+.
\end{cases}
 \end{equation}
Hence
$$
\al\wdt{\al}+\bt\wdt{\bt}=\Phi_2\al-\Phi_1\bt=:\Phi\in\calP^+
$$
and since $\Phi=\wdt{\Phi}$ it follows that $\Phi$ is constant.
\end{proof}

{\em Proof of Theorem 1}. Let $\vf(z)=\sum_{k=1}^N \gm^kz^{-k}$ be
given. We provide a constructive proof how to find $\al$ and
$\bt$. First we seek for  nontrivial polynomials
\begin{equation}
 \label{albt}
\al(z)=\sum_{k=0}^Nx_kz^k;\;\;\bt(z)=\sum_{k=0}^Ny_kz^k
 \end{equation}
which satisfy (\ref{syst}) and hence (\ref{Zmt}). If we equate all
coefficients of negative powers of $z$ of functions $\Phi_1$ and
$\Phi_2$ in (\ref{syst}) to $0$ and their $0$th coefficients to
$0$ and $1$ respectively, then we get the following system of
equations in the block matrix form
\begin{equation}
 \label{Bsyst}
\begin{cases}
\Tht X-\ol{Y}=\mbo\\
\Tht Y+\ol{X}=\mbl
\end{cases}
\end{equation}
where
$$
\Tht=\begin{pmatrix}0&\gm_{1}&\gm_{2}&\cdots&\gm_{N-1}&\gm_{N}\\
        \gm_{1}&\gm_{2}&\gm_{3}&\cdots&\gm_{N}&0\\
        \gm_{2}&\gm_{3}&\gm_{4}&\cdots&0&0\\
        \cdot&\cdot&\cdot&\cdots&\cdot&\cdot\\
        \gm_{N}&0&0&\cdots&0&0\end{pmatrix},\;
        X=\begin{pmatrix}x_0\\ x_1\\ x_2\\
        \vdots\\x_N\end{pmatrix},\;
        Y= \begin{pmatrix}y_0\\ y_1\\ y_2\\
        \vdots\\y_N\end{pmatrix},\;
        \mbo=\begin{pmatrix}0\\ 0\\0\\ \vdots\\0\end{pmatrix},\;
        \mbl=\begin{pmatrix}1\\ 0\\0\\ \vdots\\0\end{pmatrix}
$$
If we substitute the first equation of (\ref{Bsyst})
\begin{equation}
 \label{Y}
Y=\ol{\Tht}\ol{X}
 \end{equation}
into the second equation, we get
$\Tht\ol{\Tht}\,\ol{X}+\ol{X}=\mbl$, which is equivalent to
\begin{equation}
 \label{dstr}
(\ol{\Tht}\Tht+I_{N+1}){X}=\mbl.
 \end{equation}
The system (\ref{dstr}) is nonsingular as $\Tht$ is symmetric and
$\ol{\Tht}\Tht=\Tht^*\Tht$ is positive definite. (Furthermore, all
eigenvalues of $\Dt:=\ol{\Tht}\Tht+I_{N+1}$ are grater than or
equal to $1$, and hence $\|\Dt^{-1}\|\leq 1$ as well.) Hence, the
coefficients $x_k$ and $y_k$, $k=0,1,\ldots,N$, in (\ref{albt})
can be determined from (\ref{dstr}) and (\ref{Y}), and the
constructed $\al$ and $\bt$ will satisfy (\ref{Zmt}). The equation
(\ref{Const}) will be accomplished by Lemma 1 and we can achieve
(\ref{Zt}) by normalization.

If now the unitary matrix
\begin{equation}
 \label{U}
U=\begin{pmatrix} \al(1)&\bt(1)\\
-\wdt{\bt}(1)&\wdt{\al}(1)\end{pmatrix}
 \end{equation}
is not the identity matrix  (note that $\det U=1$ by virtue of
(\ref{Zt})), we can redefine $\al$ and $\bt$ by the equation
$$
\begin{pmatrix} \al&\bt\\ -\wdt{\bt}&\wdt{\al}\end{pmatrix}:=
\begin{pmatrix} \al&\bt\\ -\wdt{\bt}&\wdt{\al}\end{pmatrix}\cdot U^{-1}
$$
and thus the determined $\al$ and $\bt$ will satisfy the
conditions (\ref{Zt})-(\ref{Zmt}).

Since the determinant of the product in (\ref{Zmt}) is $1$, we
have $\al(z)\Phi_2(z)-\bt(z)\Phi_1(z)=1$ for each $z\in\bC$ (see
(\ref{syst})). Hence (\ref{0shi}) holds as well.

The uniqueness of a pair of polynomials $\al$ and $\bt$ follows
from the uniqueness of spectral factorization (see the
Introduction) since $\begin{pmatrix} 1&0\\ \vf&1\end{pmatrix}
\begin{pmatrix} \al&\bt\\ -\wdt{\bt}&\wdt{\al}\end{pmatrix}$ is the
spectral factor of $\begin{pmatrix} 1&0\\ \vf&1\end{pmatrix}
\begin{pmatrix} 1&\vf^*\\ 0&1\end{pmatrix}$. \hfill $\square$

\smallskip

Every process described during the construction of $\al$ and $\bt$
in the proof of Theorem 1 is stable under  small perturbations of
the data, which implies the validity of the following
\begin{corollary}
Let $N\geq 1$, and let $\coprod:\calP_N^-\to\calP^+_N\tm\calP^+_N$
be the map defined according to Theorem $1$ which assigns a pair
of polynomials $\al$ and $\bt$ to each $\vf\in\calP^-_N$ . Then
$\coprod$ is a continuous map.
\end{corollary}

If we take the coefficients of $\vf$ rational, then the proof goes
through without any change and the obtained coefficients of $\al$
and $\bt$ are rational as well. Thus we have the following
\begin{corollary}
If $\vf\in\calP^-_N[\bQ]$, then the corresponding polynomials
$\al$ and $\bt$ are from $\calP^+_N[\bQ]$.
\end{corollary}

\begin{theorem}
Let $N\geq 1$. For any pair of polynomials $\al,\bt\in\calP^+_N$
which satisfy $(\ref{Zt})$ and $(\ref{0shi})$ there exists a
unique $\vf\in\calP^-_N$ such that $(\ref{Zmt})$ holds.
\end{theorem}

The proof of this theorem is also constructive.

\begin{proof}
Define the function $f$ in a deleted neighborhood of $0$ as (see
(\ref{0shi}))
$$
f(z)=\begin{cases} \frac1{\al(z)}\wdt{\bt}(z)\;\;\text{ if
}\;\;\al(0)\not=0\\
-\frac1{\bt(z)}\wdt{\al}(z)\;\;\text{ if }\;\;\bt(0)\not=0
\end{cases}
$$
and let us show that
$$
\vf(z)=[f(z)]^-
$$
satisfies (\ref{Zmt}). Indeed, $\vf\in\calP^-_N$ (as
$[f(z)]^-=[f(z)]^-_N$) and
$$
\vf(z)=f(z)-[f(z)]^+
$$
in a deleted neighborhood of $0$. Consider the case $\al(0)\not=0$
(the case $\bt(0)\not=0$ can be treated analogously). Then
$$
\vf\al-\wdt{\bt}=(f-[f]^+)\al-\wdt{\bt}=\wdt{\bt}-[f]^+\al-\wdt{\bt}=-[f]^+\al
$$
and
$$
\vf\bt+\wdt{\al}=(f-[f]^+)\bt+\wdt{\al}=\frac{\wdt{\bt}\bt+\al\wdt{\al}}{\al}-[f]^+\bt=
\frac1{\al}-[f]^+\bt,
$$
which shows that the functions $\vf\al-\wdt{\bt}$ and
$\vf\bt+\wdt{\al}$ have removable singularities at $0$. Since we
know that these functions are from $\calP$, we conclude that
actually they belong to $\calP^+$. Thus (\ref{syst}) and
consequently (\ref{Zmt}) hold. Observe that
\begin{equation}
 \label{vf}
 \vf(z)=\left[\left[\frac1\al\right]^+_N\wdt{\bt}(z)\right]^-\;\;\text{
 or
 }\;\;\vf(z)=-\left[\left[\frac1\bt\right]^+_N\wdt{\al}(z)\right]^-,
\end{equation}
so that we need to compute the first $N$ coefficients of
$1/\al(z)$ or $1/\bt(z)$ in order to construct $\vf(z)$.

To show the uniqueness of $\vf(z)$ observe that, by virtue of the
Bezout theorem (see e.g. \cite[Theorem 7.6]{Mallat}), there are
$\al_0,\bt_0\in\calP^+$ such that
$\al(z)\al_0(z)+\bt(z)\bt_0(z)=1$. Hence, if $\vf\in\calP^-$
satisfies (\ref{syst}), then
$$
\al_0(z)\big(\vf(z)\al(z)-\wdt{\bt}(z)\big)+\bt_0\big(\vf(z)\bt(z)+\wdt{\al}(z)\big)
\in\calP^+
$$
and
\begin{equation}
 \label{vfm}
\vf(z)=[\al_0(z)\wdt{\bt}(z)-\bt_0(z)\wdt{\al}(z)]^-
 \end{equation}
\end{proof}

As in Theorem 1, the process of construction of $\vf$ (see
(\ref{vf})) is stable under small perturbations of $\al$ and
$\bt$. Thus we come to
\begin{corollary}
Let $N\geq 1$ and $\Om_N\sbs\calP^+_N\tm \calP^+_N$ be the set of
pairs $(\al,\bt)$ which satisfy $(\ref{Zt})$ and $(\ref{0shi})$.
Then the map $\prod:\Om_N\to\calP_N^-$ defined according to
Theorem 2, which assigns $\vf\in\calP^-_N$ to each
$(\al,\bt)\in\Om_N$, is continuous.
\end{corollary}

We can combine Corollaries 1 and 3 as follows

\begin{corollary} If $\Om^0_N\sbs\Om_n$ is the set of
$(\al,\bt)\in\Om_N$ which in addition satisfy $(\ref{Zt1})$, then
$\coprod$ is a continuous one-to-one map from $\calP^-_n$ onto
$\Om^0_N$ such that $\prod\circ\coprod=\Id_{\calP^-_N}$
\end{corollary}

\begin{proof}
If $\vf\in\calP^-_N$ and $(\al,\bt)\in\Om^0_N$ are such that
(\ref{Zmt}) holds, then $\coprod(\vf)=(\al,\bt)$ and
$\prod(\al,\bt)=\vf$.
\end{proof}

\begin{corollary} For any $(\al,\bt)\in\Om_N$,  let
$\prod(\al,\bt)=\vf$ and $\coprod(\vf)=(\al',\bt')$. Then
\begin{equation}
 \label{rec}
(\al,\bt)=(\al',\bt')U
 \end{equation}
 where the matrix $U$ is defined by the equation $(\ref{U})$.
\end{corollary}
\begin{proof}
Because of (\ref{Zt}), $U\in \calU(2)$ and $U^{-1}=U^*$. Thus
$$
\begin{pmatrix} \al&\bt\\ -\wdt{\bt}&\wdt{\al}\end{pmatrix}U^{-1}=
\begin{pmatrix} \wdt{\al}(1)\al+\wdt{\bt}(1)\bt&-\bt(1)\al+\al(1)\bt\\
-(\wdt{-\bt(1)\al+\al(1)\bt})&\wdt{\wdt{\al}(1)\al+\wdt{\bt}(1)\bt}\end{pmatrix}
$$
and $
\big(\wdt{\al}(1)\al+\wdt{\bt}(1)\bt,\,\bt(1)\al+\al(1)\bt\big)\in\Om_N^0$.
Since (see (\ref{Zmt}))
$$
\begin{pmatrix} 1&0\\ \vf&1\end{pmatrix}
\begin{pmatrix} \al&\bt\\ -\wdt{\bt}&\wdt{\al}\end{pmatrix}U^{-1}\in\calP^+(2\tm
2),
$$
we have $(\al,\bt)U^{-1}=\coprod(\vf)$. Thus
$(\al,\bt)U^{-1}=(\al',\bt')$ and (\ref{rec}) follows.
\end{proof}

\section{Description of the Method}

Using the coefficients of (\ref{Dwm}), we can define
$$
\al(z)=\frac1{\sqrt{2}}\sum_{k=0}^{N-1}a_kz^k;\;\;
\bt(z)=\frac1{\sqrt{2}}\sum_{k=0}^{N-1}b_kz^k.
$$
As has been mentioned in Section 2, the quadratic condition
(\ref{Dqc}) is equivalent to  the condition for the matrix
$$
\frac1{\sqrt{2}}A(z)=\begin{pmatrix} \al(z)&\bt(z)\\
-z^{N-1}\wdt{\bt}(z)&z^{N-1}\wdt{\al}(z)\end{pmatrix}
$$
to be  paraunitary  (see (\ref{LM}), (\ref{Pc})), and hence
(\ref{Zt}) holds, while (\ref{Dlc}) impies that
$\frac1{\sqrt{2}}A(1)=\frac1{\sqrt{2}}H_2$, where $ H_2$ is the
Haar wavelet matrix of rank $2$ defined by (\ref{H2}).

Since $(\al(z),\bt(z))\in\Om_{N-1}$, we can construct
$\vf=\prod(\al,\bt)$ according to Theorem 2, and
\begin{equation}
 \label{foll}
 (\al,\bt)=\coprod(\vf)\cdot \frac1{\sqrt{2}}H_2\;\Longrightarrow \;
 \left(\sum_{k=0}^{N-1}a_kz^k,\,\sum_{k=0}^{N-1}b_kz^k\right)=\coprod(\vf)\cdot H_2
 \end{equation}
because of Corollary 5. If we approximate $\vf$ by
$\vf_{\bQ}\in\calP^-_{N-1}[\bQ]$, $\vf\approx\vf_{\bQ}$, and
construct $\coprod(\vf_{\bQ})$, then
$\coprod(\vf_{\bQ})\in\Om_{N-1}^0(\bQ)$ by Corollary 2, and
$\coprod(\vf)\approx\coprod(\vf_{\bQ})$ by Corollary 1. Thus the
coefficients of $\coprod(\vf_{\bQ})\cdot H_2$ will be rational and
they will approximate $a_k$ and $b_k$, $k=0,1,\ldots,{N-1}$ (see
\ref{foll}). In this way, we can approximate (\ref{Dwm}) by a
matrix with rational coefficients which satisfy (\ref{Dqc}) and
(\ref{Dlc}) exactly.

The proposed method can be generalized for wavelet matrices
(\ref{wm}) of any rank $m$ since the generalization of the main
result Theorem 1 used in the method is valid for $m$-dimensional
matrices as well (see \cite[Theorem 1]{IEEE}) and at least the
formula (\ref{vfm}) for obtaining $\vf$ can be generalized as well
(see \cite{Proc1998}). However, not for any $m$, there exists a
Haar wavelet matrix of rank $m$ with rational coefficients which
would provide the generalization of formula (\ref{foll}).
Consequently, we can approximate any wavelet matrix $A$ by
$\hat{A}$ with rational coefficients for which equivalent
quadratic conditions (\ref{qc}), (\ref{Pc}) and (\ref{sf1}) hold
exactly, while (\ref{lc}) only approximately. As has been
explained in Section 2, the quadratic condition on a filter bank
is decisive for it to have the perfect reconstruction property.

\section{Computer Simulations and Results}

To construct explicitly the fractions which are close to
coefficients of Daubechies wavelet matrices $D_N$ (see
(\ref{Dwm})\,), a program was written in Mathematica 8.  A
complete screening of all possible options has been performed in
order to select the fractions with minimal denominator in the
given range. On a 2GHz Intel Core 2 Duo system with 2GB RAM
running Ubuntu 11 the calculations took less than a second. As it
was explained in preceding sections, constructed approximate
filter banks $\hat{D}_N$ have the perfect reconstruction property.
The results of different approximate computations of the
coefficients of Daubechies {\em scaling vectors} (the first rows
of $D_N$) for genus $N=2$ and $N=3$ are presented in the tables
below. The $p$th moments of these coefficients,
$M_p=\sum_{k=0}^{2N-1} h_1[k]k^p$ (see (\ref{Dpolrp})), which are
not exactly $0$ anymore because of approximation, are also
computed and located in the table. These tables are presented only
for illustrative purposes and the interested readers can produce
the different rational approximations which might be more suitable
for their specific reasons.

Table 1. $N=2$

{\fontsize{9}{9pt}\selectfont
$$
\begin{array}{| l | r@{.}l | r@{\approx}r@{.}l | r@{\approx}r@{.}l | r@{\approx}r@{.}l | }
\hline
\fbox{\rule[-2mm]{0cm}{6mm}$k=0$}&0&683012701892219&\frac{12}{17 }&0&70588&\frac{3008}{4385 }&0&68597&\frac{192000}{280913}&0&68348\\
\hline
\fbox{\rule[-2mm]{0cm}{6mm}$k=1$} &1&18301270189222&\frac{20}{17 }&1&17647&\frac{5184}{4385 }&1&18221&\frac{332288}{280913}&1&18288\\
\hline
\fbox{\rule[-2mm]{0cm}{6mm}$k=2$} &0&316987298107781&\frac{5}{17 }&0&29411&\frac{1377}{4385 }&0&31402&\frac{88913}{280913}&0&31651\\
\hline
\fbox{\rule[-2mm]{0cm}{6mm}$k=3$} &-0&183012701892219&-\frac{3}{17}&-0&17647&-\frac{799}{4385}&-0&18221&-\frac{51375}{280913}&-0&18288\\
\hline
\fbox{\rule[0mm]{0cm}{2mm}$M_1\approx$}&\multicolumn{2}{c|}{0.0}&\multicolumn{3}{c|}{0.59}&\multicolumn{3}{c|}{0.008}&\multicolumn{3}{c|}{0.001}\\
\hline
\end{array}
$$
}

Table 2. $N=3$

{\fontsize{9}{9pt}\selectfont
$$
\begin{array}{| l | r@{.}l | r@{\approx}r@{.}l | r@{\approx}r@{.}l | r@{\approx}r@{.}l | }
\hline
\fbox{\rule[-2mm]{0cm}{6mm}$k=0$}&0&470467207784164 &\frac{2888}{5249}&0&5502&\frac{2132672}{4439725 }&0&48036&\frac{2677170944}{5703228401}&0&46941\\
\hline
\fbox{\rule[-2mm]{0cm}{6mm}$k=1$} &1&14111691583144&\frac{5944}{5249}&1&1324&\frac{5059904}{4439725}&1&13968&\frac{6509075712}{5703228401}&1&14129\\
\hline
\fbox{\rule[-2mm]{0cm}{6mm}$k=2$} &0&650365000526232&\frac{3104}{5249 }&0&5913&\frac{572096}{887945 }&0&64429&\frac{3712561536}{5703228401}&0&65095\\
\hline
\fbox{\rule[-2mm]{0cm}{6mm}$k=3$} &-0&190934415568327&-\frac{1056}{5249}&-0&2011&-\frac{170688}{887945}&-0&19222&-\frac{1088205184}{5703228401}&-0&19080\\
\hline
\fbox{\rule[-2mm]{0cm}{6mm}$k=4$}&-0&120832208310396&-\frac{743}{5249 }&-0&1415&-\frac{553427}{4439725 }&-0&12465&-\frac{686504079}{5703228401}&-0&12037\\
\hline
\fbox{\rule[-2mm]{0cm}{6mm}$k=5$} &0&0498174997368838&\frac{361}{5249 }&0&0687&\frac{233261}{4439725 }&0&05253&\frac{282357873}{5703228401}&0&04950\\
\hline
\fbox{\rule[0mm]{0cm}{2mm}$M_1\approx$}&\multicolumn{2}{c|}{0.0}&\multicolumn{3}{c|}{0.256}&\multicolumn{3}{c|}{0.0357}&\multicolumn{3}{c|}{-0.0040}\\
\hline
\fbox{\rule[0mm]{0cm}{2mm}$M_2\approx$}&\multicolumn{2}{c|}{0.0}&\multicolumn{3}{c|}{1.622}&\multicolumn{3}{c|}{0.2169}&\multicolumn{3}{c|}{-0.0239}\\
\hline
\end{array}
$$
}

 \vskip+0.2cm

\ Authors' Addresses: \vskip+0.2cm

 L. Ephremidze, E. Lagvilava

\noindent  A. Razmadze mathematical Institute

\noindent I. Javakhishvili State University

\noindent 2, University Street, Tbilisi 0143, Georgia

\noindent E-mail: {\em lephremi@umd.edu; edem@rmi.ge}

\vskip+0.2cm

A. Gamkrelidze

\noindent I. Javakhishvili State University

\noindent  2, University Street, Tbilisi 0143, Georgia

\noindent E-mail: {\em alexander.gamkrelidze@tsu.ge}

\end{document}